\documentclass[12pt,a4paper,reqno]{amsart}
\usepackage{amssymb, amsmath, amsthm}

\newtheorem{theorem}{Theorem}
\newtheorem{proposition}[theorem]{Proposition}
\newtheorem{lemma}[theorem]{Lemma}

\usepackage{caption}
\usepackage{subcaption}
\usepackage{graphicx}

\newtheorem{assumption}{Assumption}

\newcommand{\R}{\mathbb{R}}

\begin{document}

\title[Asymptotic Network Independence]{Asymptotic Network Independence and Step-Size for a Distributed Subgradient Method}

\author{Alex Olshevsky}

\maketitle

\begin{abstract}%   <- trailing '%' for backward compatibility of .sty file

We consider whether distributed subgradient methods can achieve a linear speedup over a centralized subgradient method. While it might be hoped that distributed network of $n$ nodes that can compute $n$ times more subgradients in parallel compared to a single node might, as a result, be $n$ times faster,  existing bounds for distributed optimization methods are often consistent with a slowdown rather than speedup compared to a single node. 

We show that a  distributed subgradient method has this ``linear speedup'' property when using a class of square-summable-but-not-summable step-sizes which include $1/t^{\beta}$ when $\beta \in (1/2,1)$; for such step-sizes, we show that after a  transient period whose size depends on the spectral gap of the network, the method achieves a performance guarantee that does not depend on the network or the number of nodes. We also show that the same method can fail to have this ``asymptotic network independence'' property under the optimally decaying step-size $1/\sqrt{t}$ and, as a consequence, can fail to provide a linear speedup compared to a single node with $1/\sqrt{t}$ step-size. 
\end{abstract}

% \fi

\section{Introduction}

We consider the standard setting of distributed convex optimization: $f_1(x), \ldots, f_n(x)$ are convex functions from $\R^d$ to $\R$, with node $i$ of the network being the only node which can compute subgradients of the function $f_i(x)$. The goal is to compute a minimizer \begin{equation} \label{eq:mainprob} x^* \in \arg \min_{x \in \Omega} F(x),\end{equation} where $$ F(x) := \frac{1}{n} \sum_{i=1}^n f_i(x),$$ and  $\Omega$ is a closed convex set. The underlying method must be decentralized, relying only on local subgradient computations and peer-to-peer message exchanges in a certain graph $G$. In particular, we will consider the ``standard model'' of distributed optimization where at each step, node $i$ computes a subgradient of its local function, possibly performs a projection step onto the set $\Omega$, and broadcasts a message to its neighbors.  

This problem setup is a natural model for machine learning over a network of processors. Minimizing the function $F(x)$ typically comes from empirical loss minimization. The function $F(x)$ typicaly measures how well a model parametrized by the vector $x$ can fit a collection of data points; distributing the data points among $n$ processors will result in the problem formulation of Eq. (\ref{eq:mainprob}). 

A variation on this setup considers the situation when the underlying graph $G$ is taken to be the star graph, sometimes called ``local gradient descent'' (see \cite{stich2019local}). The advantage of the using the star graph is that one can design simple protocols involving rounds of interaction between the center and the leaf nodes which are not available in the setting where $G$ is an arbitrary graph. However, a disadvantage of using the star graph is that, as the number of nodes gets large, the number of bits that need to be transmitted to the center increases as well (see \cite{} which consider gradient compression to overcome this). One way to avoid this problem is to consider optimization over arbitrary graphs $G$ instead, as we do in this paper. 

This problem formulation is now classical; it first analyzed in \cite{nedic2009distributed}, where a distributed subgradient method was proposed for the unconstrained case when $\Omega=\R^n$. The case with the constraint $\Omega$ was first analyzed in \cite{nedic2010constrained}. Both papers proposed methods inspired by the ``average consensus'' literature, where nodes mix subgradient steps on their local functions with linear combinations of their neighbors iterates. 

Distributed optimization methods have attracted considerable attention since the publication of \cite{nedic2009distributed} for several reasons. First, it is hoped that distributed empirical loss minimization in machine learning could result in faster training. Second, many problems in control and signal processing among network of nodes involve  nodes acting to maximize a global objective from local information, and Eq. (\ref{eq:mainprob}) is thought to be among the simplest problems of this type. Over the past decade, thousands of  papers have been written on different variations of this problem, and it would be impossible to survey all this related work; instead, we refer the reader to the recent survey \cite{nedic2018network}. 

We next launch into a discussion of the main motivating concern of this paper, namely how the performance of distributed optimization methods compares to their centralized counterparts. We begin by discussing the available guarantees for the (centralized) subgradient method, so that we can can contrast those guarantees to the available distributed bounds in our survey of previous work, which will follow.

\subsection{The subgradient method} 

The projected subgradient method run on the function $F(x)$  takes the form
 \[ y(t+1) = P_{\Omega} \left[ y(t) - \alpha(t) g_F(t) \right], \] where $g_F(t)$ is a subgradient of the function $F(\cdot)$ at $y(t)$, and $P_{\Omega}$ is the projection onto $\Omega$. 
 
 The standard reference for an analysis of this method is the set of lecture notes \cite{boyd2003subgradient}. It is usually assumed that $||g_F(t)||_2 \leq L$ for all $t$, i.e., all subgradients are bounded; and $\Omega$ is assumed to have diameter at most $D$.  The function $F(x)$ may have more than one minimizer over $\Omega$; we select one minimizer arbitrarily and call it $x^*$.

The step-size $\alpha(t)$ needs to be properly chosen. There are two choices that are typically analyzed in this setting. One is to set $\alpha(t)=1/\sqrt{t}$, which turns out to be the optimal decay rate. The other is to choose $\alpha(t)$ to be ``square summable but not summable'' as in the following assumption.

\begin{assumption} \label{ass:sum} The sequence $\alpha(t)$ satisfies 
\begin{eqnarray*} \sum_{t=1}^{+\infty} \alpha^2(t) & < &  \infty \\ \sum_{t=1}^{+\infty} \alpha(t) & = & +\infty 
\end{eqnarray*} 
\end{assumption} 

We now briefly summarize the standard analysis of the method from \cite{boyd2003subgradient}, which the reader can consult for details. The analysis is based on the following recurrence relation, to the effect that, up to second order terms, the method gets closer to the set of minimizers at every step:  \begin{small}
\[ ||y(k+1) - x^*||_2^2 \leq   ||y(k) - x^*||_2^2 - 2 \alpha(k) (F(y(k)) - F^*) \nonumber  + L^2 \alpha^2(k), \]  \end{small} 
It is standard to re-arrange this into a telescoping sum  as
\begin{align}  2 \alpha(k) (F(y(k)) - F^*) \leq &   ||y(k) - x^*||_2^2 - ||y(k+1) - x^*||_2^2 + L^2 \alpha^2(k), \label{eq:recurr} 
\end{align} and then sum it up over $k=1, \ldots, t$. Indeed,  defining 
\[ y_{\alpha}(t) :=\frac{ \sum_{k=1}^t \alpha(k) y(k)}{\sum_{k=1}^t \alpha(k)},
\] and summing up Eq. (\ref{eq:recurr}) and  then appealing to the convexity of $F(x)$ we can obtain that 
\begin{equation} \label{eq:firstsubbound} F(y_{\alpha}(t)) - F^* 
\leq \frac{D^2 + L^2 \sum_{k=1}^t \alpha^2(k)}{2\sum_{k=1}^t \alpha(k) }, 
\end{equation} where $||y(0)-x^*||_2^2 \leq D^2$ as  $\Omega$ was assumed to have diameter $D$. Finally, by Assumption \ref{ass:sum}, the right-hand side goes to zero, and so we obtain that the subgradient method works. %We  remark again that the details of this argument can be found in any source on the subject, in particular in \cite{boyd2003subgradient}. 

A variation on this argument can get rid of the dependence on $L$ in Eq. (\ref{eq:firstsubbound}). This requires the following assumption. 

\begin{assumption} \label{ass:firstc} 
There is a constant $C_{\alpha}$ such that for all positive integers $t$,
\[ \sum_{k=1}^{t} \alpha(k) \leq C_{\alpha} \sum_{k=\lceil t/2 \rceil}^t \alpha(k). \] 
\end{assumption} 
This assumption can be motivated by observing that it is satisfied by step-sizes that decay polynomially as $\alpha(t)=1/t^{\beta}$ when $\beta>0$. 

With this assumption in place, one can set $t' = \lceil t/2 \rceil$ and instead sum  Eq. (\ref{eq:recurr}) from $t'$ to $t$. Defining the running average from time $t'$ to $t$ as
\[ y_{\alpha}'(t) :=\frac{ \sum_{k=t'}^t \alpha(k) y(k)}{\sum_{k=t'}^t \alpha(k)}
\] this immediately yields the following proposition. 

\begin{proposition} \label{prop:centsub} Suppose Assumptions \ref{ass:sum} and \ref{ass:firstc} on the step-size are satisfied, $F(x)$ is a convex functions whose subgradients are upper bounded by $L$ in the Euclidean norm, and   $t$ is large enough so that we have the upper bound
\begin{equation} \label{eq:tlower1} \sum_{k = \lfloor t/2 \rfloor}^{+\infty} \alpha^2(k) \leq  \frac{D^2}{L^2}. 
\end{equation} Then 
\[ F(y_{\alpha}'(t)) - F^* \leq 
 \frac{D^2 C_{\alpha}}{\sum_{k=1}^t \alpha(k)}.
\]  \label{prop:subgr}
\end{proposition} 
This result has no dependence on $L$, but at the expense of multiplying the dependence on $D$ by the constant $C_{\alpha}$. For example, if $\alpha(t)=1/t^{3/4}$, it is an  exercise to verify that one can take $C_{\alpha}=6$. We note that since the step-size $\alpha^2(t)$ is square summable, Eq. (\ref{eq:tlower1}) is guaranteed to hold for large enough $t$.

The bound of this proposition suggests to take $\alpha(t)$ decaying as slowly as possible (so that $\sum_{k=1}^t \alpha(k)$ grows as fast as possible) while still keeping $\alpha(t)$ square summable but not summable. There is no optimal choice, but in general one wants to pick $\alpha(t)=1/t^{\beta}$ where $\beta$ is close to $1/2$, but not $1/2$ since $\alpha(t)=1/\sqrt{t}$ is not square summable. The result will be a decay rate of $F(y_{\alpha}'(t)) - F^* = O(1/t^{1-\beta})$. 

One can redo the above argument with the rate of decay of $\alpha(t)=1/\sqrt{t}$ to obtain an optimal rate of decay. In that case, because this is not a square summable step-size, the dependence on $L$ cannot be avoided. However, since $\sum_{k=t'}^t 1/t = O(1)$, we can simply repeat all the steps above to give the bound
\begin{equation} \label{eq:normalsub} F(y_{\alpha}(t)) - F^* 
\leq O \left( \frac{D^2 + L^2}{\sqrt{t} } \right), 
\end{equation} One can also choose $\alpha(t)$ depending on the constants $D$ and $L$ to obtain better scaling with respect to those constants; however, in this paper, for simplicity we restrict our attention to unoptimized step-sizes of the form $\alpha(t)=1/t^{\beta}$. 

We next compare these results for the centralized subgradient to available convergence times in the distributed case.

\subsection{Convergence times of distributed subgradient methods} 

A number of distributed subgradient methods have been proposed in the literature, with the simplest being 
\begin{equation} \label{eq:no}  x(t+1) = W x(t) - \alpha(t) g(t), \end{equation} which was analyzed in \cite{nedic2009distributed}. Here $x(t)$ is an $n \times d$ matrix, with the $i$'th row of $x(t)$ being controlled by agent $i$; we will use $x_i(t)$ to denote the same $i$'th row. The matrix $g(t)$ is also  $n \times d$ and it's $i$'th row, which we will denote by $g_i(t)$, is a subgradient of the function $f_i(x)$ at $x=x_i(t)$. The matrix $W$ is doubly stochastic and needs to satisfy some connectivity and non-aperiodicity conditions; it suffices to assume that $W$ has positive diagonal and that the directed graph corresponding to the positive entries of $W$ is strongly connected. 

It was shown in \cite{nedic2009distributed} that, for small enough constant stepsize $\alpha(t)=\alpha$, this method results in a final error that scales linearly in $\alpha$. The projected version 
\begin{equation} \label{eq:nop} x(t+1) = P_{\Omega} \left[ W x(t) - \alpha(t) g'(t) \right], \end{equation}  was studied in \cite{nedic2010constrained}; here the projection operator $P_{\Omega}$ acts on each row of the matrix and $g'(t)$ is composed of subgradients evaluated at $Wx(t)$. It was shown that, under an appropriately decaying step-size, this scheme results convergence to an optimal solution.

Our interest is in the convergence rate of these methods; in particular, we want to see if the parallelization inherent in having $n$ nodes query subgradients at the same time  helps convergence. A useful benchmark is the consider a single node, which knows all the functions $f_i(x), i = 1, \ldots, n$, and can compute the gradient of one of these functions at every time step. We will call the rate obtained in this setup by performing full-batch subgradient descent (i.e., by computing the gradient of $F(x)$ by querying the subgradients of $f_1(x), \ldots, f_n(x)$ in $n$ steps) the {\em single-node rate}. The single node rate consists in multiplying all the rates obtained in the previous section by $n$, consistent with $n$ steps to compute a single subgradient of $F(x)$. For example, the bound of Proposition \ref{prop:subgr} becomes 
 \[ F(y_{\alpha}'(t)) - F^* \leq 
 \frac{n D^2 C_{\alpha}}{\sum_{k=1}^t \alpha(k)}. \] 
Ideally, one hopes for a factor $n$ speedup over the single note rate, since the $n$-node network can compute $n$ subgradients in parallel at every step. This corresponds to a convergence time that removes the factor of $n$ from the last equation.

Most of the existing convergence analyses do not attempt to write out all the scalings for the convergence times of distributed optimization methods; many papers write out the scaling with $t$ but do not focus on scaling with the number of nodes. Unfortunately, once those scalings are traced out within the course of the proof, they tend to scale with $(1-\sigma)^{-1}$, where $\sigma$ is the second-largest singular value associated with the matrix $W$. The quantity $(1-\sigma)^{-1}$ can scale as much as $O(n^2)$ in the worst-case over all graphs (se e\cite{nedic2018network}), so the underlying scaling could actually be worse than the single-node rate. 

A concrete example of this comes from the survey paper \cite{nedic2018network}, where a worse case rate is explicitly written out. The unconstrained case is studied, with step-size $\alpha=1/\sqrt{T}$ and the algorithm is run for $T$ steps. It is shown in \cite{nedic2018network} that 
\begin{equation} \label{eq:prevscaling} F(y_{\alpha}(t)) - F^* \leq O \left( \frac{D^2 + L^2 (1-\sigma)^{-1}}{\sqrt{T}} \right) \end{equation} Comparing this with Eq. (\ref{eq:normalsub}), we see that, in the worst case when $(1-\sigma)^{-1} \approx \Theta(n^2)$, this is a factor of $n$ slower than the single node rate -- in spite of the fact that the network can compute $n$ gradients in parallel. Similar issues affect all the upper bounds in this setting that have been derived in the previous literature, in particular the bounds derived in \cite{duchi2011dual} for dual subgradient,  in \cite{scaman2018optimal} (for the standard setting of distributed optimization where a single message exchange in neighbors is possible per step; \cite{scaman2018optimal} also explored methods where potentially $(1-\sigma)^{-1/2}$ steps of gossip are possible per step, in which case the dependence on $1-\sigma$ for the number of subgradients computed can be removed), and those implicit in  \cite{nedic2010constrained} for square-summable-but-not-summable step-sizes. 

In the paper \cite{olshevsky2017linear}, it was shown how, for a particular way to choose the matrix $W$ in a distributed way, it is possible to replace the $(1-\sigma)^{-1}$ with an $O(n)$ factor, matching the single-node rate. The idea was to use Nesterov acceleration, which allows us replace $(1-\sigma)^{-1}$ with $(1-\sigma)^{-1/2}$, and argue that for a certain particularly chosen set of weights the latter quantity is $O(n)$. However, this required slightly stronger assumptions (namely, knowing either the total number of nodes or a reasonably accurate upper bound on it). A similar idea was explored in \cite{scaman2018optimal}. While this does not offer a speedup over the single-node rate, at least it matches it. 

To attain a linear speedup over the single node rate, we would need to remove $(1-\sigma)^{-1}$ from the numerator in the scaling above. In this paper, we explore when this can be done and when it cannot.

\section{Related work} 

Among methods that converge to the optimal solution, the first examples of distributed algorithms that obtained a linear speedup were \cite{lian2017can} and \cite{morral2017success}. In \cite{lian2017can}, the case of stochastic non-convex gradient descent was considered, and it was shown that when the number of iterations is large enough, the distributed method achieves a linear speedup over the centralized. A similar result was shown in \cite{morral2017success} for the case of strongly convex distributed stochastic gradient descent; specifically, using tools from stochastic approximating, \cite{morral2017success} derived an expression for the limiting variance of the decentralized method which matches the performance of the centralized method. These papers have spawned a number of follow-up works (e.g., \cite{spiridonoff2020robust, koloskova2019decentralized, assran2019stochastic, wang2019slowmo, lian2018asynchronous, tang2018d, jiang2018linear, nazari2019dadam}, among others), making a number of further refinements in terms of the communication topology, communication requirements,  transient times, and testing on real world benchmarks.  In the earlier paper  \cite{towfic2016excess}, a linear speedup was shown under the assumptions that all the functions $f_i(\cdot)$ have the same minimum. The corresponding result in the constant step-size case, where the system converges to a neighborhood of the optimal solution, was shown even earlier in \cite{chen2015learninga, chen2015learning}, where it was proved that the limiting mean-squared error corresponding to any graph is the same as that of the complete graph. More recently, it was shown in \cite{reb2020} that distributed SGD can attain optimal statistical rates for generalization if all functions $f_i(\cdot)$ are sums of the quadratic sampled from the same distribution and the number of total samples per agent is sufficiently large.  In \cite{hendrikx2019accelerated} the case when every node knows a finite sum of functions from which it can sample was considered, and a linear speedup was shown in the case when the network was not too large. 

In contrast to our work, all of these papers studied the case of stochastic gradients, and also made stronger assumptions such as Lipschitz continuity of the gradient or strong convexity. By contrast, we study the usual (i.e., non-stochastic) subgradient method under only the assumption that the function is convex. 

This setup has been analyzed in a number of settings in the control literature, as we already discussed, and without showing any linear speedup. Additional works to be mentioned are a primal-dual approach was explored in \cite{zhu2011distributed}, using the dual subgradient method instead of the ordinary subgradient method was studied in \cite{duchi2011dual}. An analysis that elaborates on scaling with the network was given in \cite{ram2010distributed}. Finally, we mention that our paper is close in spirit to the recent works \cite{neglia2019role, neglia2020decentralized}, which is also concerned with the very same question, and gives bounds that make clear the effect of the graph topology under a number of different scenarios, in particular showcasing when it does not matter.

%Methods applying to non-identical constraints and communication delays were studied in \cite{lin2016distributed}. How finite-time convergence may be achieved (in continuous time) was analyzed in \cite{lin2016distributed}. 

%A control inspired approach based on log-barrier functions was proposed in \cite{wang2011control}. 

%In \cite{qiu2016distributed} a continuous-time approach was proposed which used the differences $P_{\Omega}(x_i(t)) - x_i(t)$ along with the subgradients as inputs to drive the system.

%A convergence time analysis was given in \cite{liu2017convergence}, but under the assumption that the function $F(x)$ is strongly convex.  

\subsection{Our contributions}

We will analyze a minor variation of Eq. (\ref{eq:no}) and Eq. (\ref{eq:nop}): 
\begin{equation} \label{eq:mainupdate1} x(t+1) =  W P_{\Omega} \left[  x(t) - \alpha(t) g(t) \right], \end{equation} 
This is slightly more natural than Eq. (\ref{eq:nop}), since $g(t)$ here is the subgradient evaluated at $x(t)$, and not at $Wx(t)$ as in Eq. (\ref{eq:nop}). This makes analysis somewhat neater. %For simplicity, we assume the initial conditions $x_i(0), i = 1, \ldots, n$ are identical and belong to $\Omega$.

Our main result will be to show that a linear speedup is achieved by this iteration on a class of square-summable-but-not-summable stepsizes which include  $\alpha(t)=1/t^{\beta}$ with $\beta \in (1/2,1)$. This is done by showing that, provided $t$ is large enough, we can give a performance bound that does not depend on $(1-\sigma)^{-1}$, i.e., is network independent. We will also show that the same assertions fail for the optimally decaying step-size $\alpha(t)=1/\sqrt{t}$. 

We next give a formal statement of our main results. First, let us state our assumptions formally as follows.

\begin{assumption} \label{ass:functions}
Each function $f_i(x): \R^d \rightarrow R$ is a convex with all of its subgradients bounded by $L$ in the Euclidean norm. Moreover, the set $\Omega$ is a closed convex set. Each node begins with an identical initial condition $x_i(0) \in \Omega$.
\end{assumption} 

\begin{assumption} \label{ass:matrix} The matrix $W$ is nonnegative, doubly stochastic, and with positive diagonal. The directed graph corresponding to the positive entries of $W$ is strongly connected. 
\end{assumption} 

Secondly, we will be making an additional assumption on step-size.   

\begin{assumption} The sequence $\alpha(t)$ is nonincreasing and there is a constant $C_{\alpha'}$ such that for all $t$,  \label{ass:secondc}
\[ \alpha(\lfloor t/2 \rfloor) \leq C_{\alpha}' \alpha(t). \]
\end{assumption}  This assumption essentially bounds how fast
$\alpha(t)$ can decrease over the period of $t/2, \ldots, t$.  It is motivated by observing that it holds for step-sizes of the form $\alpha(t)=1/t^{\beta}$.

Finally, let us introduce the notation 
\[ \overline{x}(t) = \frac{1}{n} \sum_{i=1}^n x_i(t),\] for the average of the iterates at time $t$. We adopt a general convention that, for a vector or a matrix, putting an overline will mean referring to the average of the rows. 

Similarly to what was done in the previous subsection, we  define 
\[ x_{\alpha}'(t) =\frac{ \sum_{k=t'}^t \alpha(k) \overline{x}(k)}{\sum_{k=t'}^t \alpha(k)}
\] 

Our first main result is the following theorem.

\begin{theorem}[Asymptotic Network Independence with Square Summable Step-Sizes] Suppose Assumptions \ref{ass:sum}, \ref{ass:firstc}, and \ref{ass:secondc} on the step-size, Assumption \ref{ass:functions} on the functions, and Assumption \ref{ass:matrix} on the mixing matrix $W$ all hold.  \label{thm:main}
 
Then if $t$ is large enough so that the tail sum satisfies the upper bound \begin{equation} \label{eq:alphabound1} \sum_{k=\lfloor t/2 \rfloor}^{+\infty} \alpha^2(t) \leq \frac{ D^2 (1-\sigma)}{10C_{\alpha}' L^2} \end{equation} and also 
\begin{equation} \label{eq:tlower12} t \geq \Omega \left( \frac{1}{1-\sigma} \log  \left[ (1-\sigma) t \alpha_{\rm max}/(C_{\alpha}' \alpha(t)) \right] \right) 
\end{equation} then we have the network-independent bound \begin{equation} \label{eq:firstneti} F(x_{\alpha}'(t)) - F^* \leq \frac{ D^2  C_{\alpha}}{ \sum_{k=0}^t \alpha(k)} \end{equation}
In particular, if $\alpha(t) = 1/t^{\beta}$ where $\beta$ lies in the range $(1/2,1)$, then when $t$ additionally satisfies 
\[ t^{2 \beta-1} \geq \Omega_{\beta} \left( \frac{L^2}{D^2 (1-\sigma)} \right) \] we have the network-independent bound \begin{equation} \label{eq:secondneti} 
F(x_{\alpha}'(t)) - F^* \leq O_{\beta} \left( 
\frac{D^2}{t^{1-\beta}} 
\right),\end{equation} where the subscript of $\beta$ denotes that the constants in the $O(\cdot)$ and $\Omega(\cdot)$ notation depend on $\beta$. 
\end{theorem} 

At the risk of being repetitive, we note that the performance guaranteed by this theorem is asymptotically network independent, as the only dependence on the spectral gap $1-\sigma$ is in the transient. The point of this theorem is to contrast Eq. (\ref{eq:firstneti}) with Proposition (\ref{prop:centsub}). The two guarantees are identical, which implies that the distributed optimization method analyzed in Theorem \ref{thm:main} gives us a linear time speedup over the single-node rate using the same-step size. Likewise, Eq. (\ref{eq:secondneti}) gives a network-independent bound (though, again, the size of the transient until it holds depends on the network), and can be thought of as a linear-time speedup over the corresponding single-node rate.

Such linear speedups are significant in that they provide a strong motivation for distributed optimization: one can claim that, over a network with $n$ nodes, the distributed optimization is $n$ times faster than a centralized one,  at least provided $t$ is large enough. Without such speed-ups, it is much more challenging to motivate the study of distributed optimization in the first place.

Note that the lower bound of Eq. (\ref{eq:tlower12}) on the transient is not fully explicit, as $t$ actually appears on both sides. However, for large enough $t$ the inequality always holds. Informally, this is because $t$ on the right-hand side is inside the logarithm. Slightly more formally, it is immediate that, if the bound of Eq. (\ref{eq:tlower12}) fails to hold, then $\alpha(t) \leq C_1 t e^{-C_2 t}$ for all $t$ where  $C_1,C_2$ depend on $1-\sigma$,$\alpha_{\rm max}$, and $C_{\alpha}$; and this  would contradict the non-summability of $\alpha(t)$ in Assumption \ref{ass:sum}. 

Finally, note that this theorem contains bounds on the performance achieved by $x_{\alpha}'(t)$, which a particular kind of running average across the iterates of all the nodes in he network. This can be tracked in a distributed way by the nodes if the number of iterations $T$ to be performed is known ahead of time: each node can take a running average of its own iterates and, after $T$ iterations have been performed, the nodes do a run of average consensus to compute the average of these running averages. The algorithm from \cite{olshevsky2017linear} takes $O(n \log (1/\epsilon))$ iterations to come $\epsilon$-close to the average on any undirected $n$-node graph, so this incurs an extra cost which does not depend on $T$ (moroever, that consensus method along with the method from \cite{rebeschini2017accelerated} takes $O(D)$ steps on many common graphs of diameter $D$, e.g., grids). It is also possible to modify this idea to handle  the case when the number of iterations to be done is not known ahead of time (e.g., each node can restart computation of the running average when the number of iterations is a power of $2$) at the cost of losing of a constant multiplicative factor in the error $F(\cdot)-F^*$.

\bigskip

\bigskip

Unfortunately, the previous theorem does not apply to $\beta=1/2$ which, as discussed earlier, is the best rate of decay for the subgradient method. 
In fact, our next theorem will show something quite different occurs when $\beta=1/2$: we will construct a counterexample where the distributed method has network-dependent performance regardless of how large $t$ is. 

We next give a formal statement of this result. Our first step is to describe how we will choose the matrix $W$ depending on the underlying graph. Let us adopt the convention that, given an undirected graph $G=(V,E)$ without self-loops, we will define the symmetric stochastic matrix $W_{G, \epsilon}$ as 
\[ [W_{G,\epsilon}]_{ij} = \begin{cases} \epsilon & (i,j) \in E \\ 0 & \mbox{ else} \end{cases}, \] and we set diagonal entries $[W_G]_{ii}$ to whatever values result in a stochastic matrix. Clearly, $\epsilon$ should be strictly smaller than the largest degree in $G$.

We next define $G_n'$ to be a graph on $2n$ nodes obtained as follows: two complete graphs on nodes $u_1, \ldots, u_n$ and $v_1, \ldots, v_n$ are joined by connecting $u_i$ to $v_i$. We note that because the largest degree in this graph is $n$, any $\epsilon$ used to construct a stochastic $W_{G,\epsilon}$ should be upper bounded by $1/n$. 

Our second main result shows that when we run Eq. (\ref{eq:mainupdate1}) on this graph $G_n'$, then with an appropriate choice of functions we will never obtain a performance independent $\epsilon^{-1}$; and since, as we just, $\epsilon^{-1}$ grows as $\Omega(n)$, the performance will always scale with $n$ no matter how long we wait. 

\begin{theorem}[Lack of Asymptotic Network Independence with $1/\sqrt{t}$ Step-Size] \label{thm:supp} Consider the distributed optimization method of Eq. (\ref{eq:mainupdate1}) with 
\begin{itemize} \item The functions 
\[ f_i(x) = \gamma |x |, \]  when $i \in \{u_1, \ldots, u_n\}$ and  
\[ f_i(x) = \frac{1}{2} |x-1|,\] for   $i \in \{v_1, \ldots, v_n\}$
\item Step-size $\alpha(t) = 1/\sqrt{t}$.
\item Constraint set $\Omega = [-a,a]$.
\item Initial conditions $x_i(0)=0$. \end{itemize}

Then, there exists a choice of the constants $\gamma > 1$ and $a$ (appearing in the above definitions), both  independent of $n$ or $\epsilon$, such that, on the graph $G_n'$ with $n \geq 4$ and any choice of $\epsilon \leq 1/n$,   there exists an infinite sequence $g_i(t)$ such that the quantity $x(t)$ defined through Eq. (\ref{eq:mainupdate1}) satisfies: 
\begin{itemize} 
\item For $i \in \{v_1, \ldots, v_n\}$, $x_i(t) - x^*$ is a nonnegative sequence that does not depend on $i$ and satisfies
\begin{equation} \label{eq:decaylowerbound} x_i - x^* = \Omega \left( \frac{\epsilon^{-1}}{\sqrt{t}} \right), \mbox { for all } i \in \{ v_1, \ldots, v_n\}, \end{equation}  for all large enough $t$. 
\item $x_i(t) = x^*$ for all $i \in \{u_1, \ldots, u_n\}$ and all $t$. 
\end{itemize} 
\end{theorem} 

To parse the statement of this theorem, note that there is some freedom to choose the subgradient of functions like $|x|$ at $|x-1|$ at zero and one, respectively. Consequently, the theorem has to assert that there is a choice of subgradients such that the solution of Eq. (\ref{eq:mainupdate1}) has the desired behavior. 

To put the guarantees of this theorem into context, observe that, as a consequence of Eq. (\ref{eq:decaylowerbound}), not only does the average  $(1/n) \sum_{i=1}^n x_i(t) - x^*$ scale as $O(\epsilon^{-1}/\sqrt{t})$ but so does any convex combination of this quantity over the various $t$'s (in particular, the quantities $x_{\alpha}(t)$ or $x_{\alpha}'(t)$ discussed earlier). It immediately follows that $F(\cdot)-F^*$ for all of these quantities also scales linearly with $\epsilon^{-1}$, i.e., 
\begin{align}
F(x_{\alpha}(t)) - F^* & \geq  \Omega \left( \frac{\epsilon^{-1}}{\sqrt{t}} \right) \label{eq:fdecay1} \\ 
F(x_{\alpha}'(t)) - F^* & \geq  \Omega \left( \frac{\epsilon^{-1}}{\sqrt{t}} \right) \label{eq:fdecay2}
\end{align}
Moreover, since $\epsilon \leq 1/n$, the performance of Eq. (\ref{eq:mainupdate1}) with $1/\sqrt{t}$ step-size does not attain a speedup over the corresponding single-node rate.
This is to be contrasted with Eq. (\ref{eq:firstneti}) and Eq. (\ref{eq:secondneti}) which, provided one  waits long enough, attain a linear speedup over the single-node rate. 

In fact, it is easy to see that, for the matrix $W_{G_n'}$, the inverse spectral gap will grow with $\epsilon^{-1}$, so that the presence of $\epsilon^{-1}$ in Eq. (\ref{eq:fdecay1})  and Eq. (\ref{eq:fdecay2}) is equivalent to scaling with $(1-\sigma)^{-1}$.  For completeness, we give a proof of this assertion  later in this paper as Proposition \ref{prop:gap}. 

The main result of this paper is the contrast between Theorems \ref{thm:main} and Theorem \ref{thm:supp}. In the former case, we have asymptotic network independence and a linear speedup whenever the step-size is $1/t^{\beta}$ when $\beta>1/2$. Unfortunately, the latter theorem shows that setting $\beta=1/2$ can ruin  this. 

We remark that the last theorem has some similarities with Eq. (75) of \cite{neglia2020decentralized}, which considers the speed at which a decentralized optimization method can move towards infinity when no minimizer exists under a constant step-size, and finds it can be network-dependent. 

Finally, we remark that a visual illustration of this is given in Section \ref{sec:simul}, and the interested reader can skip ahead to see what the difference in performance looks like in a simulation of this example between $1/\sqrt{t}$ step-size and $1/t^{\beta}$ step-size where $\beta \in (1/2,1)$.

%Finally, we note that since all the degrees in the graph were $\Theta(n)$, the would apply to other ways to design a doubly stochastic matrix on a graph. For example, the choice of the ``Metropolis matrix,'' which is the unique stochastic matrix satisfying $W_{ij}=1/\max(d(i),d(j))$ whenever $i$ and $j$ are connected (and zero otherwise) would lead to the same matrix as in the above theorem. 

\subsection{Organization of this paper} We give a proof of Theorem \ref{thm:main} in Section \ref{sec:firstproof} and a proof of Theorem \ref{thm:supp} in Section \ref{sec:secondproof}. A simulation of our counterexample to network independence with $1/\sqrt{t}$ step-size is given in Section \ref{sec:simul}, which shows what the contrast between the presence of absence of network independence looks like numerically. 

Our results raise the possibility that while a more slowly decaying step-size might be better for the centralized subgradient method, the opposite might be true in the distributed case (note that the transient bound of Theorem \ref{thm:main} improves with higher $\beta$). We show that this kind of ``step size inversion'' does indeed occur on a class of randomly generated problems in Section \ref{sec:inversion}. Finally, our concluding Section \ref{sec:concl} mentions several open problems and future directions. 

\section{Proof of Theorem \ref{thm:main}\label{sec:firstproof}}
   
In this section, we provide a proof of Theorem \ref{thm:main}. Our first step is to rewrite Eq. (\ref{eq:mainupdate1}) in a way that will be easier to analyze. We set $s(t)$ to be the ``gradient mapping'' defined as
\[ s(t) := \frac{x(t) -   P_{\Omega} \left[  x(t) - \alpha(t) g(t) \right] }{\alpha(t)} \] so that  Eq. (\ref{eq:mainupdate1}) can be written as
\begin{equation} \label{eq:mainupdate2} x(t+1) = W  \left[ x(t) - \alpha(t) s(t) \right]  
\end{equation} 
Consistent with our previous notation, we will use $s_i(t)$ to denote the $i$'th row of the matix $s(t)$. 

In this formulation, we  no longer have to explicitly deal with the projection, which is incorporated into the definition of $s(t)$. As we will see, the quantity $s(t)$, which is typically known as the ``gradient mapping'' in the case where the functions are smooth, has some properties similar to the properties of the a subgradient. Although we suspect that this is well-known, we have been unable to find a reference; to our knowledge, in the current literature, various properties of $s(t)$ are only listed in the case where the functions $f_i(\cdot)$ are smooth, which is assumption we do not make here.  The next lemma is our first statement to this effect, showing that $s_i(t)$ inherits any upper bound on $g_i(t)$.

\begin{lemma} If $||g_i(t)||_2 \leq L$ then  $||s_i(t)||_2 \leq L$. \label{lemm:supper}
\end{lemma} 

\begin{proof} We first observe that for all $i,t$ we have that $x_i(t) \in \Omega$. Indeed, $x_i(t)$ is obtained as a convex combination of vectors projected onto $\Omega$ and so itself belongs to $\Omega$ by convexity.  
We then use this, along with the fact that projection onto convex sets is nonexansive, to argue that 
\begin{align*}
    ||s_i(t)||_2 & = \frac{|| x_i(t) -P_{\Omega} \left[ x_i(t) - \alpha(t) g_i(t) \right]||_2}{|\alpha(t)|} \\ 
        & = \frac{|| P_{\Omega} \left[ x_i(t) \right] -P_{\Omega} \left[ x_i(t) - \alpha(t) g_i(t) \right]||_2}{|\alpha(t)|} \\
        & \leq \frac{||\alpha(t) g_i(t)||_2}{|\alpha(t)|} \\ 
        & \leq ||g_i(t)||_2 \\ 
        & \leq L.
\end{align*} 
\end{proof} 

Next, we note that it is standard that the subgradient $g_i(t)$ of the convex function $f_i(x)$ at $x_i(t)$ satisfies the relation
\begin{equation} \label{eq:standard_subgr} g_i(t) (x_i(t) - x^*)^T \geq f(x_i(t)) - f_i(x^*). \end{equation} 
Our next lemma shows that $s_i(t)$ satisfies a similar inequality up to a ``higher order'' term. 
\begin{lemma}   Under Assumption \ref{ass:functions}, we have that for all $i=1, \ldots, n$, \[ \alpha(t) s_i(t) (x_i(t) - x^*)^T \geq \alpha(t) f(x_i(t)) - f_i(x^*) - \frac{\alpha^2(t)}{2} L^2.\] \label{lemm:s}
\end{lemma}

Note that whereas Eq. (\ref{eq:standard_subgr}) does not contain the step-size $\alpha(t)$, Lemma \ref{lemm:s} does. This is because the definition of the quantity $s_i(t)$ contained the step-size $\alpha(t)$ (unlike the subgradient $g_i(t)$ which  is obviously defined independently of step-size). 

\smallskip

\begin{proof}
We start from the relation
\[ x_i(t) - \alpha(t) s_i(t) =  P_{\Omega} \left[ x_i(t) - \alpha(t) g_i(t) \right], \] which is just a rearrangement of the definition of $s(t)$. Our next step is to subtract $x^*$ and take the squared Euclidean norm of both sides.  On the left-hand side, we have 
\[ ||x_i(t) - x^*||^2 - 2 \alpha(t) s_i(t) (x_i(t) - x^*)^T + \alpha^2(t) ||s_i(t)||_2^2. \] 
On the right-hand side, we use the fact that projecting onto $\Omega$ cannot increase Euclidean distance from $x^*$ to obtain an upper bound of \[  ||x_i(t) - x^*||_2^2 - 2 \alpha(t) g_i(t) (x_i(t) - x^*)^T + \alpha^2(t) ||g_i(t)||_2^2  \] Putting these two facts together, we obtain the inequality 
\begin{align*} - 2 \alpha(t) s_i(t) (x_i(t) - x^*)^T \leq & - 2 \alpha(t) g_i(t) (x_i(t) - x^*)^T  + \alpha^2(t) ||g_i(t)||_2^2 
\end{align*} or 
\begin{align*}  2 \alpha(t) s_i(t) (x_i(t) - x^*)^T \geq &  2 \alpha(t) g_i(t) (x_i(t) - x^*)^T   - \alpha^2(t) ||g_i(t)||_2^2 
\end{align*} Now using Assumption \ref{ass:functions} and Eq. (\ref{eq:standard_subgr}), we obtain 
\[ 2 \alpha(t) s_i(t)^T (x_i(t) - x^*) \geq 2 \alpha(t) (f_i(x_i(t)) - f_i(x^*)) - \alpha^2(t) L^2,\] which proves the lemma. 
\end{proof}

The final lemma we will need bounds the distance between each $x_i(t)$ and $\overline{x}(t)$  as $O(\alpha(t))$ (where the constant inside this $O(\cdot)$-notation will depend on the matrix $W$). Such bounds are  standard in the distributed optimization literature.
%, which typically proceed by arguing that the average value $\overline{x}(t)$ approximately performs gradient descent, which would require arguing that all the $x_i(t)$ are close to $\overline{x}(t)$.

We introduce some new notation which we will find convenient to use. 
We will use ${\bf 1}$ to denote the all-ones vector in $\R^n$, so that ${\bf 1} \overline{x}(t)$ has the same dimensions as $x(t)$. We adopt the notation $\sigma$ to denote the second-largest singular value of the matrix $W$; under Assumption \ref{ass:matrix}, we have that $\sigma < 1$ while the largest singular value is $1$ corresponding to the all-ones vector\footnote{Formally, this is implied from Lemma 4 of \cite{nedic2009distributed2}. That lemma implies that $||Wx||_2 \leq ||x||_2$ and, under Assumption \ref{ass:matrix}, equality holds if and only if $x$ is a multiple of ${\bf 1}$.}. In the sequel, for a vector $y$ we will use the inequality
\begin{equation} \label{eq:wineq} ||W(y - \overline{y})||_2^2 \leq \sigma^2 ||y-\overline{y}||_2^2 \leq \sigma^2 ||y||_2^2. 
\end{equation}

With these preliminaries in place, we have the following lemma.

\begin{lemma} Suppose Assumptions \ref{ass:sum}, \ref{ass:firstc}, and \ref{ass:secondc} on the step-size, Assumption \ref{ass:functions} on the functions, and Assumption \ref{ass:matrix} on the mixing matrix $W$ all hold.  When $$ t \geq \Omega \left( \frac{1}{1-\sigma} \log  \frac{(1-\sigma) t \alpha_{\rm max}}{C_{\alpha}' \alpha(t) } \right) $$ we have that 
\[ ||x(t) - {\bf 1} \overline{x} ||_F \leq \frac{2 C_{\alpha}' \alpha(t) L \sqrt{n}}{1-\sigma}. \] \label{lemm:distance}
\end{lemma} 

\begin{proof} Recall that, by assumption, the initial conditions are identical; let us denote them all by $x_1$. Thus starting from Eq. (\ref{eq:mainupdate2}) we have that 
\[ x(t) = W^{t-1} x_1 - \alpha(1) W^{t-1} s(1) - \cdots  - \alpha(t-1) W s(t-1) \] we can use the fact that multiplication by a doubly stochastic matrix doesn't affect the mean of a vector to obtain that 
\[ \overline{x}(t) = x_1 - \alpha(1) \overline{s}(1) - \cdots - \alpha(t-1) \overline{s}(t-1). \] Next, using the ``MATLAB notation''  $[A]_{:,i}$ for the $i$'th column of a matrix $A$, we can apply  Eq. (\ref{eq:wineq}) in the following sequence of equations:
\begin{align*}||x(t) - {\bf 1} \overline{x}(t)||_F & \leq  \sum_{k=1}^{t-1} \alpha(k) ||W^{t-k} (s(k) - {\bf 1} \overline{s}(k))||_F \\
& = \sum_{k=1}^{t-1} \alpha(k) \sqrt{\sum_{i=1}^n||W^{t-k} [s(k) - {\bf 1} \overline{s}(k)]_{:,i}||_2^2} \\  
& \leq  \sum_{k=1}^{t-1} \alpha(k) \sqrt{\sum_{i=1}^n \sigma^{2(t-k)} || [s(k)]_{:,i}||_2^2} \\
& \leq \sum_{k=1}^{t-1} \alpha(k) \sigma^{t-k} ||s(k)||_F. 
\end{align*} Let us break the last sum at $t'=t-\lceil t/2 \rceil$ and bound each of the two pieces separately. The first piece, over the range $t=1, \ldots, t'$ is bounded simply using the fact that all subgradients are upper bounded by $L$ in the Euclidean norm (and consequently $||s(t-k)||_F \leq L \sqrt{n})$; whereas the second piece, over the last $t/2$ steps, is upper bounded as a geometric sum. 
The result is
\[ ||x(t) - {\bf 1} \overline{x}(t)||_F \leq  \frac{t}{2}  \alpha_{\rm max} L  \sqrt{n} \sigma^{\lceil t/2 \rceil} + \frac{L \sqrt{n}}{1-\sigma} \alpha(\lfloor t/2 \rfloor). \] We next use that 
$x^{1/(1-x)} \leq e^{-1}$ when $x \in [0,1]$ as well as Assumption \ref{ass:secondc} to obtain that 
\begin{align*}  ||x(t) - {\bf 1} \overline{x}(t)||_F \leq  t \alpha_{\rm max} L \sqrt{n}  e^{-t (1-\sigma)/2} + \frac{L \sqrt{n} C_{\alpha}' \alpha(t)}{1-\sigma}.
\end{align*} 
When $t \geq \Omega((1-\sigma)^{-1} \log \left[ (1-\sigma) (t \alpha_{\rm max}/(C_{\alpha}' \alpha(t)) \right]$, the first term is upper bounded by the second and the lemma is proved.  

\end{proof} 

With all these lemmas in place, we are now ready to give a proof of our first main result. 

\begin{proof}[Proof of Theorem \ref{thm:main}]
Starting from Eq. (\ref{eq:mainupdate2})  we obtain
\begin{align*}
    \overline x(t+1) - x^*= \overline{x}(t) - \alpha(t) \overline{s}(t) - x^* 
\end{align*}
so that
\begin{align*}
    ||\overline{x}(t+1) - x^*||_2^2  = &   ||\overline{x}(t) - x^*||_2^2 + \alpha^2(t) ||\overline{s}(t)||^2  - 2 \alpha(t) \overline{s}(t) (\overline{x}(t) - x^*)^T \\ 
     = &  ||\overline{x}(t) - x^*||_2^2 + \alpha^2(t) ||\overline{s}(t)\||_2^2  - 2 \alpha(t) \left(\frac{1}{n} \sum_{i=1}^n s_i(t) (\overline{x}(t) - x^*)^T \right) \\
     = &  ||\overline{x}(t) - x^*||_2^2 + \alpha^2(t) ||\overline{s}(t)||_2^2 - 2 \alpha(t) \left(\frac{1}{n} \sum_{i=1}^n s_i(t) (x_i(t) - x^*)^T \right)  \\ & + 2 \alpha(t) \left(\frac{1}{n} \sum_{i=1}^n s_i(t) (x_i(t) - \overline{x}(t))^T \right) \\
     \leq &  ||\overline{x}(t) - x^*||_2^2 +  \alpha^2(t) L^2   - 2 \alpha(t) \frac{1}{n} \sum_{i=1}^n f_i(x_i(t)) - f_i(x^*) 
     \\ & + L^2 
     \alpha^2(t) + 2 \alpha(t)  \frac{1}{n} \sum_{i=1}^n L ||x_i(t) - \overline{x}(t)||_2, 
     \end{align*} where, in the above sequence of inequalities, we used Lemma \ref{lemm:supper} to bound the norm of $||\overline{s}(t)||_2^2$, Lemma \ref{lemm:s} to bound $s_i(t)^T (x_i(t) - x^*)$, and Cauchy-Schwarz in the very last step. Now using the fact that each $f_i(\cdot)$ is $L$-Lipschitz, which follows from Assumption \ref{ass:functions},
     we have 
     \begin{align}
       ||\overline{x}(t+1) - x^*||_2^2  \leq  ||\overline{x}(t) - x^*||_2^2 + 2 \alpha^2(t) L^2 &  - 2 \alpha(t) \frac{1}{n} \sum_{i=1}^n f_i(\overline{x}(t)) - f_i(x^*) \nonumber  \\ & + 4 \alpha(t) L \frac{1}{n} \sum_{i=1}^n ||x_i(t) - \overline{x}(t)||_2. \label{eq:disagreement}
\end{align} 

We next bound the very last term in the sequence of inequalities above. Our starting point is the observation that  
\[ \sum_{i=1}^n ||x_i(t) - \overline{x}(t)||_2 \leq \sqrt{n} ||x(t) - \overline{x}(t)||_F,\] which follows by an application of Cauchy-Schwarz.  We then use Lemma \ref{lemm:distance} to bound the right-hand side. This yields that, for $t$ large enough to satisfy the assumptions of that lemma, 
\begin{align*} ||\overline{x}(t+1) - x^*||_2^2
\leq &  ||\overline{x}(t) - x^*||_2^2 + 2 \alpha^2(t) L^2  - 2 \alpha(t) \frac{1}{n} \sum_{i=1}^n f_i(\overline{x}(t)) - f_i(x^*)  + 4 \alpha(t) L \frac{2 C_{\alpha}' \alpha(t) L} {1-\sigma}, 
\end{align*} implying that 
\begin{align*} 2 \alpha(t) \left[ F(\overline{x}(t)) - F^* \right] 
\leq &  ||\overline{x}(t) - x^*||_2^2 - ||\overline{x}(t+1) - x^*||_2^2  +  2 \alpha^2(t) L^2    + 8 \alpha^2(t)  \frac{ C_{\alpha}' L^2} {1-\sigma}, 
\end{align*} As before, let $t' = \lfloor t/2 \rfloor$. We sum the last inequality up frome time $t'$ to time $t$ to obtain  
\begin{align*} 
2 \sum_{k=t'}^t \alpha(k) \left[ F(\overline{x}(k)) - F^* \right] \leq ||\overline{x}(t') - x^*||_2^2 +  \frac{10 C_{\alpha}' L^2}{1-\sigma} \sum_{k=t'}^t \alpha^2(k),
\end{align*} where we used that $C_{\alpha}' \geq 1$ (because $\alpha(t)$ is nonincreasing) and that $\sigma < 1$ to combine the terms involving $\alpha^2(t)$. 

Dividing both sides by $2 \sum_{k=t'}^t \alpha(t)$ and using convexity of $F(x)$, we obtain 
\begin{equation} \label{eq:penultimate} 
F(\overline{x}_{\alpha}(t)) - F^*\leq \frac{||\overline{x}(t') - x^*||_2^2}{2 \sum_{k=t'}^t \alpha(k)}  +  \frac{10 C_{\alpha}' L^2}{1-\sigma} \frac{\sum_{k=t'}^t \alpha^2(k)}{2 \sum_{k=t'}^t \alpha(k)}
\end{equation} Now because $\alpha(t)$ is square summable, we have that 
\[ \lim_{t \rightarrow \infty} \sum_{k=t'}^t \alpha^2(k) \leq \lim_{t \rightarrow \infty} \sum_{k=\lfloor t/2 \rfloor}^{+\infty} \alpha^2(k) = 0. \] In particular, Eq. (\ref{eq:alphabound1}) will eventually be satisfied, and when 
that happens, the second term of Eq. (\ref{eq:penultimate}) will be upper bounded by the first. We will then have \[ F(\overline{x}_{\alpha}(t)) - F^*\leq \frac{D^2}{\sum_{k=t'}^t \alpha(k)}  
\] The first part of the theorem, namely Eq. (\ref{eq:firstneti}), now follows immediately. 

Next, we suppose that $\alpha(k)=1/k^{\beta}$ where $\beta \in (1/2,1)$. This step-size satisfies all the assumptions we have made; however, starting from Eq. (\ref{eq:penultimate}), we can write down some more effective bounds. Indeed, by the usual method of upper/bounding sums by the corresponding integrals, we have that
\begin{align*}
    \sum_{k=t'}^t \frac{1}{k^\beta} \geq & \Omega \left( \left( 1-\frac{1}{2^{-\beta + 1}} \right) \frac{t^{-\beta + 1}}{-\beta+1} \right)  = \Omega_{\beta} \left( t^{-\beta + 1}\right) \\
    \sum_{k=t'}^t \left(  \frac{1}{k^\beta} \right)^2 \leq & O \left( \left( 1-\frac{1}{2^{-2\beta + 1}} \right) \frac{t^{-2\beta + 1}}{-2\beta+1}
    \right) = O_{\beta} \left( t^{-2 \beta + 1} \right), 
\end{align*} where the subscript of $\beta$ denotes that the constant could depend on $\beta$. Plugging this into Eq. (\ref{eq:penultimate}), we have \begin{align}
F(\overline{x}_{\alpha}(t)) - F^*
& \leq O_{\beta} \left( \frac{D^2}{t^{1-\beta}}
+ \frac{L^2 t^{-2 \beta + 1}}{(1-\sigma)t^{-\beta+1}} 
\right) \nonumber \\ 
& \leq O_{\beta} \left( \frac{D^2}{t^{1-\beta}}
+ \frac{L^2}{(1-\sigma)t^{\beta}} 
\right) \label{eq:twoterms}
\end{align}
Therefore when  
\[ t^{2 \beta-1} \geq \Omega_{\beta} \left( \frac{L^2 }{D^2 (1-\sigma)} \right) \] we have that the first term of Eq. (\ref{eq:twoterms}) dominates and 
\begin{align*}
F(\overline{x}_{\alpha}(t)) - F^* \leq O_{\beta} \left( 
\frac{D^2}{t^{1-\beta}} 
\right)
\end{align*} 
This proves Eq. (\ref{eq:secondneti}) and the proof of the theorem is now complete. 
\end{proof} 

\section{Proof of Theorem \ref{thm:supp} \label{sec:secondproof}}
 The proof below will analyze an explicit example of a graph where the dependence on spectral gap never disappears, no matter how large $t$ is.  The  graph is $G_n'$, which is a graph on $2n$ vertices $\{u_1, \ldots, u_n\} \cup \{v_1, \ldots,v_n\}$ defined shortly before the statement of Theorem \ref{thm:supp}, and the underlying matrix $W$ is $W=W_{G_n', \epsilon}$, defined in the same place. 
 
 We begin with a brief description of the intuition underlying the counter-example. The high-level idea is that the performance of distributed subgradient descent can be thought of in terms of a recursion that moves towards the optimal solution, perturbed by some error due to network disagreement. Indeed, this is the form taken by the proof of Theorem \ref{thm:main}: we wrote down a recursion satisfied by $F(\overline{x}(t)) - F^*$, and the updates of that recursion featured terms that depended on the network-wide  disagreement $\sum_{i=1}^n ||x_i(t) - {\bf 1} \overline{x}(t)||$ (see Eq. (\ref{eq:disagreement}) above). 
 
 In the absence of any disagreement -- i.e., if we could magically set $\sum_{i=1}^n ||x_i(t) - {\bf 1} \overline{x}(t)||=0$ at every step -- the recursion of Eq. (\ref{eq:disagreement}) will converge to the optimal solution at a rate $F(\cdot) - F^* = O(1/\sqrt{t})$. Now if the disagreement decays faster than $O(1/\sqrt{t})$, it is intuitive that it has no effect on the ultimate convergence rate. On the other hand, if the disagreement decays as $\sim 1/\sqrt{t}$, then it may well dominate. Furthermore,  it should also be intuitive that the term $\sum_{i=1}^n ||x_i(t) - {\bf 1} \overline{x}(t)||$ can only be bounded in terms of the spectral gap (since, in effect, this term measures how successful repeated multiplication by the matrix $W$ is in driving all the nodes closer together), and this way the spectral gap will appear in the final performance of the method. 
 
 In short, we need to argue that the recursion of Eq. (\ref{eq:disagreement}) is, in some sense, ``tight'' when the step-size is $1/\sqrt{t}$.  Unfortunately, we know of no pleasant way to make this argument. The only way we are able to do this is to come up with an example where we can write down an explicit formula for the solution of Eq. (\ref{eq:no}) at any time $t$, and this is in effect what we do below on the graph $G_n'$. The graph $G_n'$ is particularly well-suited for the purpose because it is symmetric:  we can write down a solution where all $x_{u_i}(t)$ have the same value, and all $x_{v_i}(t)$ have the same value. Unfortunately, even in this simple case, the argument turns out to be somewhat involved. 
 \bigskip

 We next turn to the proof itself. The key ingredient will be the following technical lemma.

\begin{lemma}
Consider the update rule determined by $y(1)=0$ and 
\begin{small}
\begin{equation} \label{eq:yupdate} y(t+1) = \left( 1 - \epsilon \right) y(t) - \frac{   (1/2) (1-\epsilon) {\rm sign}(y(t)-1) + \epsilon \Delta(t) }{\sqrt{t}},
\end{equation} 
\end{small} where 
\[ \Delta(t) = \frac{\epsilon \sqrt{t} y(t) - (\epsilon/2)  {\rm sign}(y(t)-1)}{1-\epsilon}. \] Here ${\rm sign}(x)$ is the  function which returns $1$ when $x \geq 0$ and $-1$ otherwise. 

Then when $\epsilon \in (0,1/4]$, we have that $y(t) \in [0,2]$ for all $t$;  and furthermore,
\begin{equation} \label{eq:ysqrtbound}
y(t) \leq  \frac{ 2\epsilon^{-1}}{\sqrt{t}},
\end{equation} for all $t$. Finally, for all $t$ larger than some $t_1$, we also have 
\begin{equation} \label{eq:concl}  y(t) \geq \frac{\epsilon^{-1}}{16 \sqrt{t}}. \end{equation} \label{lemm:recurr}
\end{lemma}

We postpone the proof of this lemma for now, as we think the reader will be more interested in it once it is seen how the recursion of Eq. (\ref{eq:yupdate}) naturally appears in the analysis of distributed subgradient descent on the graph $G_n'$. Thus, we will first give a proof of Theorem \ref{thm:supp} which relies on this lemma, and then we will go back and supply a proof of the lemma.

\bigskip

\begin{proof}[Proof of Theorem \ref{thm:supp}] We argue that 
\begin{equation} \label{eq:traj} 
\begin{cases} x_i(t) = 0 & i \in \{  u_1, \ldots, u_n \} \\ 
x_i(t) = y(t) & i \in \{v_1, \ldots, v_n\} 
\end{cases}
\end{equation} is a valid trajectory of Eq. (\ref{eq:no}) with $W=W_{G_n',\epsilon}$ where, recall, $W_{G_n',\epsilon}$ is defined shortly before the statement of Theorem \ref{thm:supp}. Here $y(t)$ is defined in the statement of  Lemma \ref{lemm:recurr} and by ``valid trajectory'' we mean that there exists a sequence of vectors $g_i(t)$, with $g_i(t)$ being a valid subgradient of $f_i(x)$ at $x_i(t)$, resulting in our main update of Eq. (\ref{eq:no}) taking the values specified in Eq. (\ref{eq:traj}) for all $t$. 

Our proof of this, given next,  will depend on a particular choice of  the constants $\gamma>1$ and $a$ from the statement of Theorem \ref{thm:supp}; these constants will be defined in the course of the proof. Once the validity of Eq. (\ref{eq:traj}) is proved, Theorem \ref{thm:supp} is immediate, conditional on Lemma \ref{lemm:recurr}. Indeed, it is easy to see that because $\gamma > 1 > 1/2$, we have that $x^*=0$; and Eq. (\ref{eq:concl}) then provides the lower bound claimed in statement of Theorem \ref{thm:supp}. 

Our proof of the validity of Eq. (\ref{eq:traj}) is by induction. At time $t=1$, we just have $x_i(t)=0$, so there is nothing to prove. 
Suppose Eq. (\ref{eq:traj}) is a valid trajectory over times $1, \ldots, t$, and let us consider time $t+1$. Our first step is to argue that, for an appropriately large choice of the constant $a$, we can simply omit the projection step in Eq. (\ref{eq:no}). 

Indeed, observe that by definition of the functions $f_i(\cdot)$ (see statement of Theorem \ref{thm:supp}), we have that the subgradients are in $[-\gamma, \gamma]$ for $i \in \{u_1, \ldots, u_n\}$, and in $[-1/2,1/2]$ for $i \in \{v_1, \ldots, v_n\}$; and recall that later we will specify $\gamma > 1$. Using Lemma \ref{lemm:recurr},  which asserts that $y(t) \in [0,2]$ for all $t$, we have that $x_i(t) \in [0,2]$ for all $i,t$; and what  follows from all this that   $x_i(t) - \alpha(t) g_i(t) \in [-\gamma, 2+\gamma]$ for all $i,t$. 

To be able to omit the projection step from Eq. (\ref{eq:mainupdate1}), it suffices to have all $x_i(t) - \alpha(t) g_i(t)$ be in the interior of $[-a,a]$. But since we have just argued that $x_i(t) - \alpha(t) g_i(t) \in [-\gamma, 2+\gamma]$ for all $i,t$, we see that it suffices to choose e.g.,  $a = 3+ \gamma$. 

The update of Eq. (\ref{eq:no}) then becomes
\[ x(t+1) = W_{G_n',\epsilon} \left[ x(t) - \alpha(t) g(t) \right]. \] Our next step is to work out what this gives for nodes   $i \in \{ u_1, \ldots, u_n\}$ given the particular form of $W_{G_n',\epsilon}$.  Since all $g_{u_i}(t), i = 1, \ldots, n$ are subgradients of the same function evaluated at zero, we will be considering the case when they are the same. In that case, we have that \begin{small} 
\begin{align*}
    x_{u_i}(t+1) & = (1-n \epsilon) x_{u_i}(t) + (n-1) \epsilon u_i(t) + \epsilon x_{v_i}(t) - \alpha(t) [W_{G_n',\epsilon} g(t)]_{u_i} \\ 
    & = x_{u_i}(t) + \epsilon   \left( x_{v_i}(t) - x_{u_i}(t) \right) - \frac{(1-\epsilon)g_{u_i}(t) + \epsilon g_{v_i}(t)}{\sqrt{t}}. 
\end{align*} \end{small}  Multiplying both sides by $\sqrt{t}$ and using that 
$x_{v_i}(t) = y(t), x_{u_i}(t) =0$ by the inductive hypothesis, we obtain
\[ \sqrt{t} x_{u_i}(t+1) = \epsilon \sqrt{t} y(t) - (1-\epsilon) g_{u_i}(t) - \epsilon g_{v_i}(t). \]

In order to show that Eq. (\ref{eq:traj}) holds at time $t+1$ for the node $u_i$, we need to have  $x_{u_i}(t+1) = 0$. The last equation allows us to see what is needed for this to be the case; setting the left-hand side to zero, we obtain 
\begin{equation} \label{eq:subgrtimet} g_{u_i}(t) = (1-\epsilon)^{-1} \left( \epsilon \sqrt{t} y(t) - \epsilon g_{v_i}(t) \right).
\end{equation} Our argument shows that as long we ``select'' this number as the local subgradient chosen by all the nodes $u_i$ at time $t$, then Eq. (\ref{eq:traj}) holds at time $t+1$ for the nodes $u_i$. But is this number a valid choice of subgradient for all the functions  $f_i(\cdot)$ at the point $0$?  

Observe that, by Lemma \ref{lemm:recurr}, we have that $y(t) \leq 2 \epsilon^{-1}/\sqrt{t}$ for all $t$. Consequently, $\epsilon \sqrt{t} y(t)  \in [0,2]$. Since $g_{v_i}(t) \in [-1/2,1/2]$ and $\epsilon \leq 1/4$ (because we assumed $n \geq 4$ and $\epsilon \leq 1/n$), we have that the expression in parenthesis on right-hand side of Eq. (\ref{eq:subgrtimet}) lies between $-1/2$ and $17/8$. Using $\epsilon \leq 1/4$ again, we see that the right-hand side of Eq. (\ref{eq:subgrtimet}) is at most $17/6$ in absolute value.  Thus, as long as we choose $\gamma$ sufficiently big, e.g., $\gamma=3$,  the choice of $g_{u_i}(t)$ in Eq. (\ref{eq:subgrtimet}) is a valid subgradient. 

To recap, we have just shown that Eq. (\ref{eq:traj}) holds at time $t+1$ for at least for the nodes $u_i$ by specifying what the local subgradients at nodes $u_i$ should  be at time $t$.  For this choice of subgradient to be valid, we had to choose $\gamma=3$ and also $a=3+\gamma = 6$. We next argue that, with these choices of $\gamma, a$ and $g_{u_i}(t)$, we can also choose a number to return as the local subgradient of all the nodes $v_i$ such that Eq. (\ref{eq:traj}) holds at time $t+1$ for the nodes $v_i$ as well.

Indeed, by induction we have that 
for $i \in \{v_1, \ldots, v_n\}$, 
\begin{align}
x_{v_i}(t+1) & = x_{v_i}(t) + \epsilon (0 - x_{v_i}(t)) - \frac{(1-\epsilon) g_{v_i}(t)+ \epsilon g_{u_i}(t)}{\sqrt{t}}, \label{eq:vrecur}
\end{align} where we used that $x_{u_i}(t)=0$ by the inductive hypothesis, and $x_{v_i}(t)=x_{v_j}(t)$ for all $i,j$, also by the inductive hypothesis.  We want to show that there is a choice of $g_{v_i}(t)$ that turns the left-hand side into $y(t+1)$. But the choice  $g_{v_i}(t) = (1/2) {\rm sign}(y(t)-1)$  is valid and turns the pair of Eq. (\ref{eq:vrecur} and Eq. (\ref{eq:subgrtimet})  into exactly Eq. (\ref{eq:yupdate}), so it certainly results in $x_{v_i}(t+1)=y(t+1)$. 

To summarize, we have shown how to choose valid subgradients at each step so that Eq. (\ref{eq:mainupdate1}) turns into the recursion relation satisfied by Eq. (\ref{eq:traj}). The proof is now complete. 
\end{proof}

It remains to prove Lemma \ref{lemm:recurr}. 

\bigskip

\begin{proof}[Proof of Lemma \ref{lemm:recurr}] We first argue that $y(t) \in [0,2]$ for all $t$. This will follow from the following three assertions. 

First: that   $y(t)$ decreases whenever it is above one.  Indeed, from Eq. (\ref{eq:yupdate}), we have that whenever $y(t) \geq 1$, 
\[ y(t+1) \leq (1-\epsilon) y(t) + \frac{\epsilon^2/2}{1-\epsilon} = y(t) - \epsilon \left( y(t) - \frac{\epsilon/2}{1-\epsilon} \right). \] Since we have assumed $\epsilon \in (0,1/4]$ and $y(t) \geq 1$, the expression in parenthesis is positive and $y(t+1) < y(t)$. 

Second: that $y(t)$ cannot decrease below zero. Indeed, if $y(t) \in [0,1)$, then from Eq. (\ref{eq:yupdate}), 
\begin{align*} y(t+1) & \geq (1-\epsilon) y(t) + \frac{(1-\epsilon)/2}{\sqrt{t}} - \frac{\epsilon^2 \sqrt{t} y(t)}{\sqrt{t} (1-\epsilon)} - \frac{\epsilon^2/2}{(1-\epsilon) \sqrt{t}} \\ & = \left( 1 - \epsilon - \frac{\epsilon^2}{1-\epsilon} \right) y(t) + \frac{(1-\epsilon)/2}{\sqrt{t}}- \frac{\epsilon^2/2}{(1-\epsilon) \sqrt{t}} \\
& \geq 0,
\end{align*} where the last step follows because $\epsilon \leq 1/4$ implies that 
\[ 1 - \epsilon - \frac{\epsilon^2}{1-\epsilon} \geq 0 \] and 
\[ \frac{1-\epsilon}{2} \geq \frac{\epsilon^2/2}{1-\epsilon}. \] 
On the other hand, if $y(t) \geq 1$, then
\begin{align*} y(t+1) & \geq (1-\epsilon) y(t) - \frac{(1-\epsilon)/2}{\sqrt{t}} - \frac{\epsilon^2 \sqrt{t} y(t)}{\sqrt{t}(1-\epsilon)} \\ 
& = \left( 1 - \epsilon - \frac{\epsilon^2}{1-\epsilon} \right) y(t)  - \frac{(1-\epsilon)/2}{\sqrt{t}}  \\ 
& \geq \frac{1}{2} y(t) - \frac{1/2}{\sqrt{t}} \\ 
& \geq \frac{1}{2} - \frac{1}{2} = 0,
\end{align*} where we used that $\epsilon \leq 1/4$ again. Thus, regardless of whether $y(t) \in [0,1)$ or $y(t) \geq 1$, we have that $y(t+1) \geq 0$. 

Third: that if  $y(t) \in [0,1)$, then the increase to the next step is  at most $1/(2\sqrt{t})$.   Indeed, for $y(t) \in [0,1)$ we have from Eq. (\ref{eq:yupdate}), 
\begin{align*} y(t+1) & \leq \left( 1 - \epsilon - \frac{\epsilon^2}{1-\epsilon} \right) y(t) + \frac{(1-\epsilon)/2}{\sqrt{t}} \\
& \leq y(t) + \frac{1}{2 \sqrt{t}}
\end{align*} 

Putting the three assertions above together, we obtain that $y(t) \in [0,2]$ for all $t$. 

\bigskip

We next prove that $y(t)$ decays as $O(\epsilon^{-1}/\sqrt{t})$. Indeed, since $|{\rm sign}(y(t)-1)| \leq 1$, we have that,
\[ y(t+1) \leq \left(1-\epsilon - \frac{\epsilon^2}{1-\epsilon} \right) y(t) + \frac{1}{2\sqrt{t}}.\] Defining $z(t)$ via  
\begin{equation} \label{eq:zdef} z(t+1) = (1-\epsilon) z(t) + \frac{1}{2\sqrt{t}},\end{equation} with $z(1)=0$, we have that $y(t) \leq z(t)$. To prove an upper bound on $y(t)$,  we simply need to establish the same upper bound for $z(t)$.

First, we  argue that 
\begin{equation} \label{eq:zbound}  z(t) \leq \sum_{k=1}^{t-1} \frac{1}{2 \sqrt{k}} = \frac{1}{2} + \sum_{k=2}^{t-1} \frac{1}{2 \sqrt{k}} \leq \frac{1}{2} +  \int_1^{t-1} \frac{1}{2 \sqrt{u}} ~du  \leq \sqrt{t} \end{equation} 
Next, we multiply both sides of Eq. (\ref{eq:zdef}) by  $\sqrt{t+1}$ to obtain 
\[ \sqrt{t+1} z(t+1) = (1-\epsilon) \sqrt{t+1} z(t) + \frac{1}{2} \sqrt{\frac{t+1}{t}},\] and now using concavity of square root we obtain 
\[ \sqrt{t+1} z(t+1) \leq (1-\epsilon) \left( \sqrt{t} + \frac{1}{2\sqrt{t}} \right) z(t) + \frac{1}{2} \sqrt{\frac{t+1}{t}}. \] Using Eq. (\ref{eq:zbound}), this implies that 
\[ \sqrt{t+1} z(t+1) \leq (1-\epsilon) \sqrt{t} z(t) + 2,\] which gives that $\sqrt{t} z(t) \leq 2/\epsilon$. This proves that  $y(t) \leq 2 \epsilon^{-1}/\sqrt{t}$ and concludes the proof of Eq. (\ref{eq:ysqrtbound}).

It only remains to prove Eq. (\ref{eq:concl}). Since, from Eq. (\ref{eq:ysqrtbound}) which we have just established,  we  know that $y(t) \rightarrow 0$, it follows that, for all $t$ larger than some $t_0$, ${\rm sign}(y(t)-1)=-1$; and therefore, for such $t$, 
\[ y(t+1) = \left(1-\epsilon - \frac{\epsilon^2}{1-\epsilon} \right) y(t) + \frac{(1-\epsilon)-\epsilon^2/(1-\epsilon)}{2\sqrt{t}}, \] 
which, because $\epsilon \in (0,1/4]$, implies that 
\[ y(t+1) \geq (1-2\epsilon) y(t) + \frac{1}{4 \sqrt{t}}.\] 
This means that for all $t \geq t_0+1$
\[ y(t) \geq \sum_{k=t_0}^{t-1} \frac{1}{4 \sqrt{k}} (1 - 2 \epsilon)^{t-(k+1)} \geq \frac{1}{4 \sqrt{t}} \sum_{k=t_0}^{t-1} (1-2 \epsilon)^{t-(k+1)}.
\] As $t \rightarrow \infty$, the geometric sum in the final term above approaches $1/(2 \epsilon)$. It follows that, when $t$ is bigger than some $t_1$, it is larger than half of that, so that 
\[ y(t) \geq \frac{\epsilon^{-1}}{16 \sqrt{t}}.\] This completes the proof. 
\end{proof}

Finally, we put Theorem \ref{thm:supp} into context by arguing that the inverse spectral gap of the matrix $W_{G_n',\epsilon}$  is $\Theta \left( \epsilon^{-1} \right)$ when $\epsilon$ is small enough. This explains that scaling with $\epsilon^{-1}$ is the same scaling with the inverse spectral gap of $W_{G_n',\epsilon}$. 

\begin{proposition} \label{prop:gap} Let $1=\lambda_1, \lambda_2, \ldots, \lambda_n$ be the eigenvalues of the matrix $W_{G_n',\epsilon}$ arranged in descending order (because $W_{G_n', \epsilon}$ is symmetric, its eigenvalues are real). Then 
\[ \lambda_2(G_n) = 1 - 2 \epsilon. \] 
\end{proposition} 

\begin{proof} The main idea is that it is quite easy to diagonalize $W_{G_n',\epsilon}$ explicitly. Indeed, let $L_{K_n}$ be the Laplacian of the complete graph on $n$ vertices and let us adopt the notation $I_n$ for the $n \times n$ identity matrix. Then it is immediate that 
\[ W_{G_n, \epsilon} = I_{2n} 
- \epsilon
\left(
\begin{array}{cc} 
L_{K_n} & 0 \\ 
0 & L_{K_n} 
\end{array} \right) 
- \epsilon \left(
\begin{array}{cc} 
I_n & -I_n \\ 
-I_n & I_n 
\end{array} \right) \]
 Now it is well known how to diagonalize the Laplacian of the complete graph: $L_{K_n} {\bf 1} = 0, L_{K_n} x = n x$ for any vector $x$ orthogonal to ${\bf 1}$. It follows that the eigenvectors of $W_{G_n',\epsilon}$ can be written out as follows. Pick $n-1$ linearly independent vectors $x_1, \ldots, x_{n-1}$ whose entries sum to zero, and observe that the collection of vectors $\{ [x_i,x_i]^T, [x_i,-x_i]^T ~|~ i = 1, \ldots, n-1 \} \cup \{  [{\bf 1}, {\bf 1}]^T, [{\bf 1}, -{\bf 1}]^T\}$ is an orthogonal basis for $\R^{2n}$.

We next verify that all of these vectors are in fact eigenvectors of $W_{G_n,\epsilon}$. Indeed:
\[ W_{G_n',\epsilon} \left( \begin{array}{c} {\bf 1} \\ {\bf 1}
\end{array} \right)  =  \left( \begin{array}{c} {\bf 1} \\ {\bf 1} 
\end{array} \right), 
\]
\[ W_{G_n',\epsilon} \left( \begin{array}{c} {\bf 1} \\ -{\bf 1} 
\end{array} \right)  = (1-2\epsilon) \left( \begin{array}{c} {\bf 1} \\ -{\bf 1} 
\end{array} \right), 
\] and for any $x$ orthogonal to the all-ones vector, 
\[ W_{G_n',\epsilon} \left( \begin{array}{c} x \\ x 
\end{array} \right)  = (1-n\epsilon) \left( \begin{array}{c} x \\ x 
\end{array} \right), 
\] while 
\[ W_{G_n',\epsilon} \left( \begin{array}{c} x \\ -x 
\end{array} \right)  =  (1-n\epsilon -2 \epsilon) \left( \begin{array}{c} x \\ -x 
\end{array} \right). 
\] Thus the set of eigenvalues of $W_{G_n',\epsilon}$ is $\{1, 1-2\epsilon, 1-n\epsilon, 1-(n+2)\epsilon\}$. This completes the proof. \end{proof}

\section{What does network independence look like? Simulating the counterexample. \label{sec:simul}}  We next give a numerical illustration of what network independence looks like. Specifically, we simulate the example constructed in the proof of Theorem \ref{thm:supp}. Our main purpose in doing so is to show how network independence, and its lack, are very stark phenomena that can be instantly ``read off'' the simulation results. 

First, we found that our choice of $\gamma$ and $a$ in the course of the proof were somewhat conservative; numerically, we find that we can choose the slightly smaller values  $\gamma = 2$ and $a=5$. We simulated the step-size of $1/t^{\beta}$ for two choices of $\beta$. Specifically, the  step-size choice of $\beta=1/2$ is shown in Figure \ref{fig:sub1} while the choice of $\beta=3/4$ is shown in Figure \ref{fig:sub2}.  Each simulation shows three different values of $n$.

It is crucial to note that {\em the y-axis of both figures shows $t^{1-\beta} (F(\overline{x}(t) - F^*)$.} Intuitively, choosing a step-size of $\alpha(t) = 1/t^{\beta}$ will result in error $F(x(t)) - F^*$ that decays like $1/t^{1-\beta}$. What we are interested is whether the constant in front of this depends on the network, so we rescale the error to make this clear.

Comparing Figure \ref{fig:sub1} and Figure \ref{fig:sub2} illustrates our main result. In Figure \ref{fig:sub1}, we have argued in Theorem \ref{thm:supp} that network independence does not occur. In particular, the underlying network has spectral gap that grows linearly with $n$ (see Proposition \ref{prop:gap}) and, as a result, the performance of the method ultimately behaves like $\sim n/\sqrt{t}$. Thus the effect of $n$ is never  forgotten.

On the other hand, in Figure \ref{fig:sub2}, we see a peak whose height/length may depend on the spectral gap, and on $n$, in some fashion; but eventually, every curve drops below $1$. In other words, we have that, eventually,  $t^{1-\beta} (F(\overline{x}(t) - F^*) \leq 1$. We see that, after a transient periodi, the performance  satisfies an $O(t^{-(1-\beta)})$ decay bound that does not depend on $n$ or the spectral gap. Of course, here one must include the usual caveat that the size of the transient until this happens will depend on the spectral gap, and thus on $n$ in graph families where the spectral gap depends on $n$.

\begin{center}
\begin{figure}
\centering
\begin{subfigure}{2.4in}
  \centering
  \includegraphics[width=2.25in]{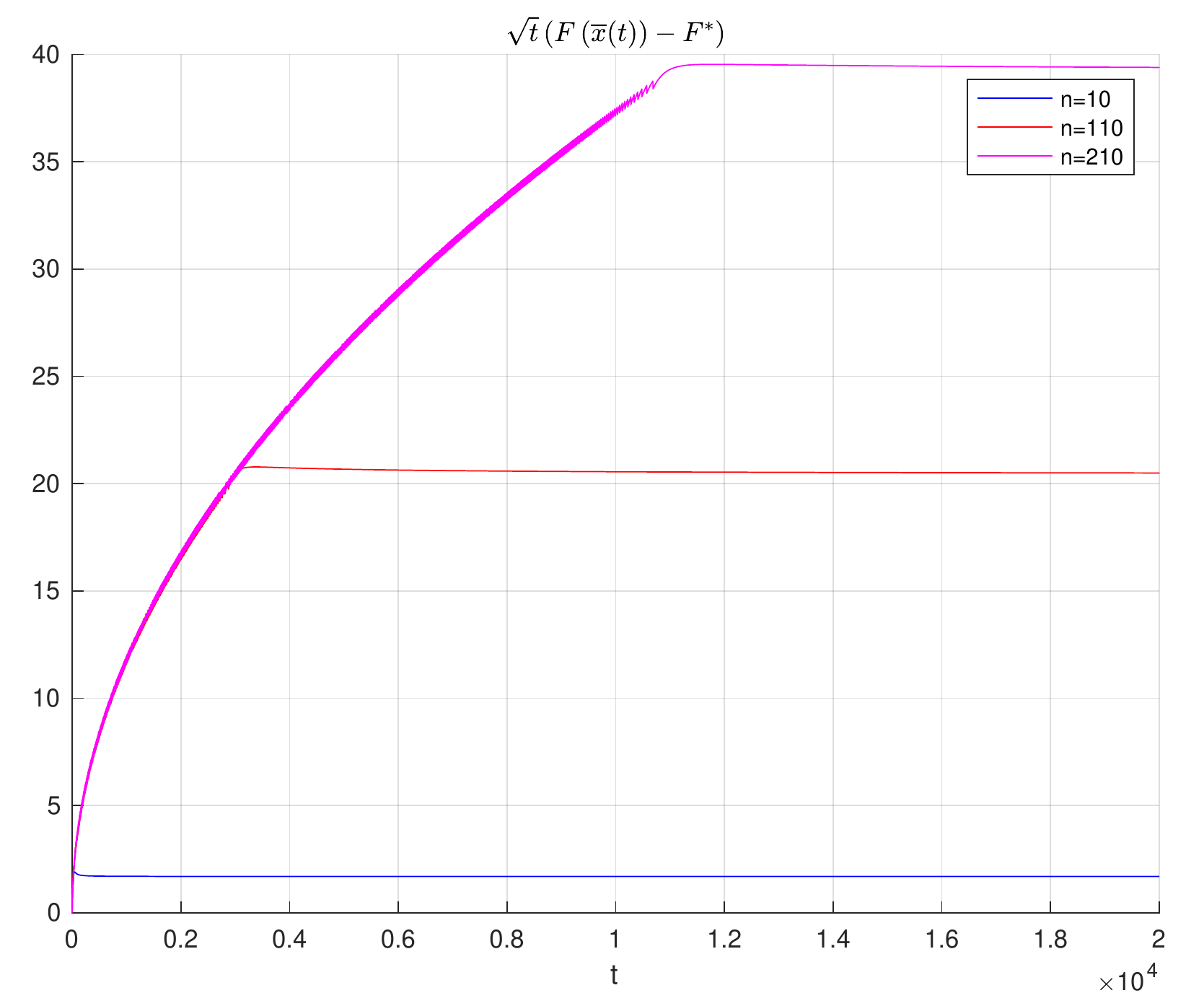}
  \caption{Step-size $\alpha(t)=1/\sqrt{t}$.}
  \label{fig:sub1}
\end{subfigure}%
\begin{subfigure}{2.4in}
  \centering
  \includegraphics[width=2.25in]{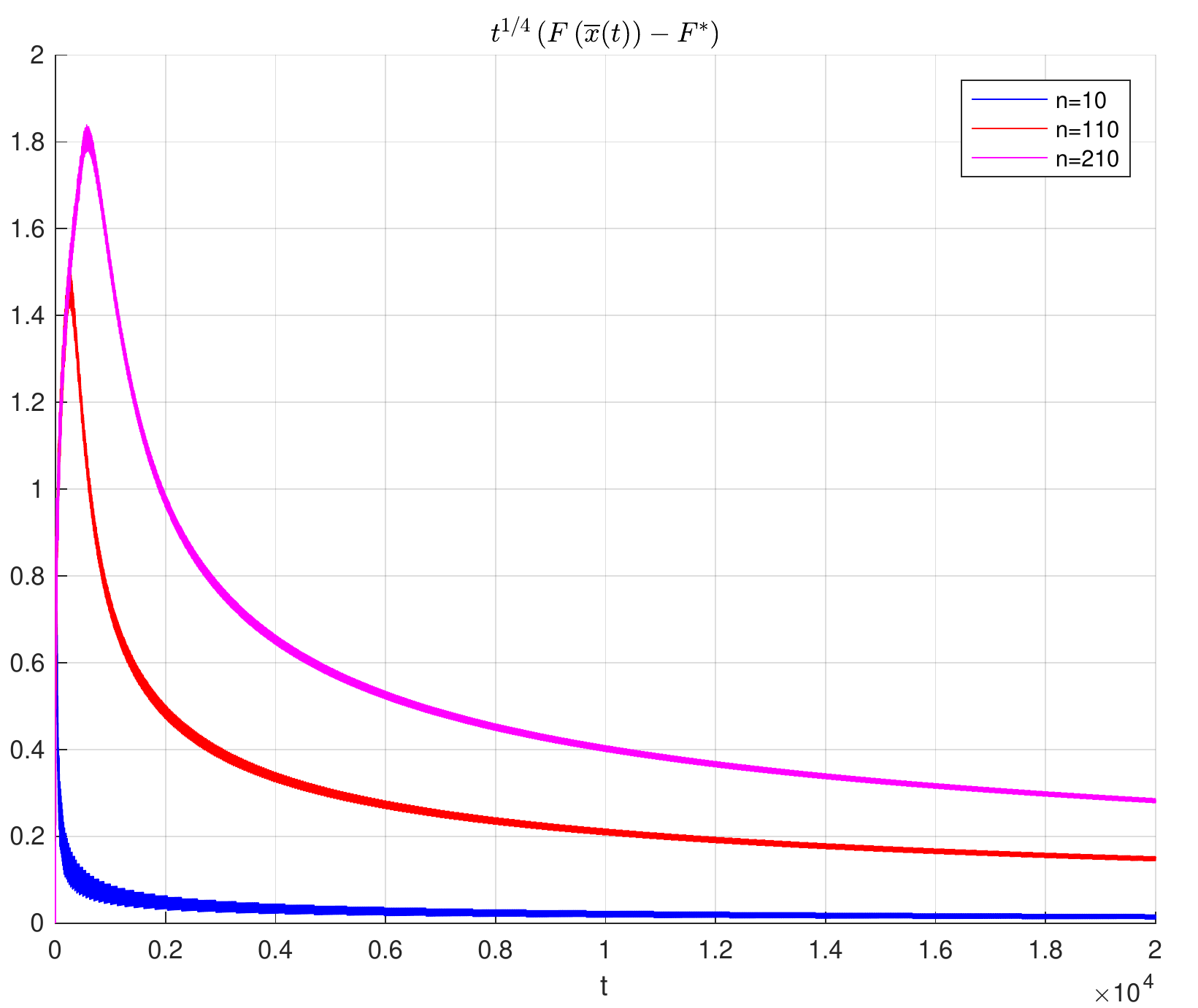}
  \caption{Step-size $\alpha(t)=1/t^{3/4}$}
  \label{fig:sub2}
\end{subfigure}
\caption{Simulation of the counter-example constructed in the proof of Theorem \ref{thm:supp} for two different choices of step-size. The figure on the right is consistent with asymptotic network independence, while the figure on the left is consistent with its absence.}
\label{fig:test}
\end{figure}
\end{center}

\section{Step-size inversion\label{sec:inversion}}

The contrast between Theorems \ref{thm:main} and \ref{thm:supp} raises an interesting question: could it help in practice to choose a more slowly decaying step-size in distributed optimization? 

Asymptotically, the answer is clearly no. Choosing a step-size that decays like $1/t^{\beta}, \beta \in [1/2,1)$ results in an error $F(\cdot) - F^*$ that decays like $1/t^{1-\beta}$, so a lower $\beta$ is  better if we wait long enough. But this is an asymptotic statement, and it says nothing about what happens in practice when we count iterations until we are close to the optimal solution. 

Informally, we may think of the distributed subgradient method as having two sources of error: the first kind, which is error that comes from the asymptotic convergence rate; and the second kind, which is the error that comes purely from the network disagreement effects. Theorems \ref{thm:main} and Theorem \ref{thm:supp} suggest the second source of error decays faster for larger $\beta$ (this is why the transient size given in Theorem \ref{thm:main} decreses with larger $\beta$). Even though asymptotically, the first source of error dominates, if we only count iterations until we get to a certain neighborhood of  the optimal solution, it might be that the second kind of error dominates in such a ``non-asymptotic'' experiment. 

We next give evidence that this is sometimes the case. Specifically, we describe a class of problems with random data, and simulate the centralized subgradient method; as expected,  performance (measured as number of iterations until the gradient mapping $s(t)$ is small) gets worse as we increase $\beta$ in step-size $1/t^{\beta}$. We then simulate the distributed subgradient method and observe a more complicated relationship between step-size and performance; in particular, performance could improve as we increase $\beta$ over some range of values. We call this phenomenon step-size inversion (since the effect of increasing $\beta$ could be opposite for centralized and distributed methods).

\begin{figure}
\centering
\begin{subfigure}{2.7in}
  \centering
  \includegraphics[width=2.5in]{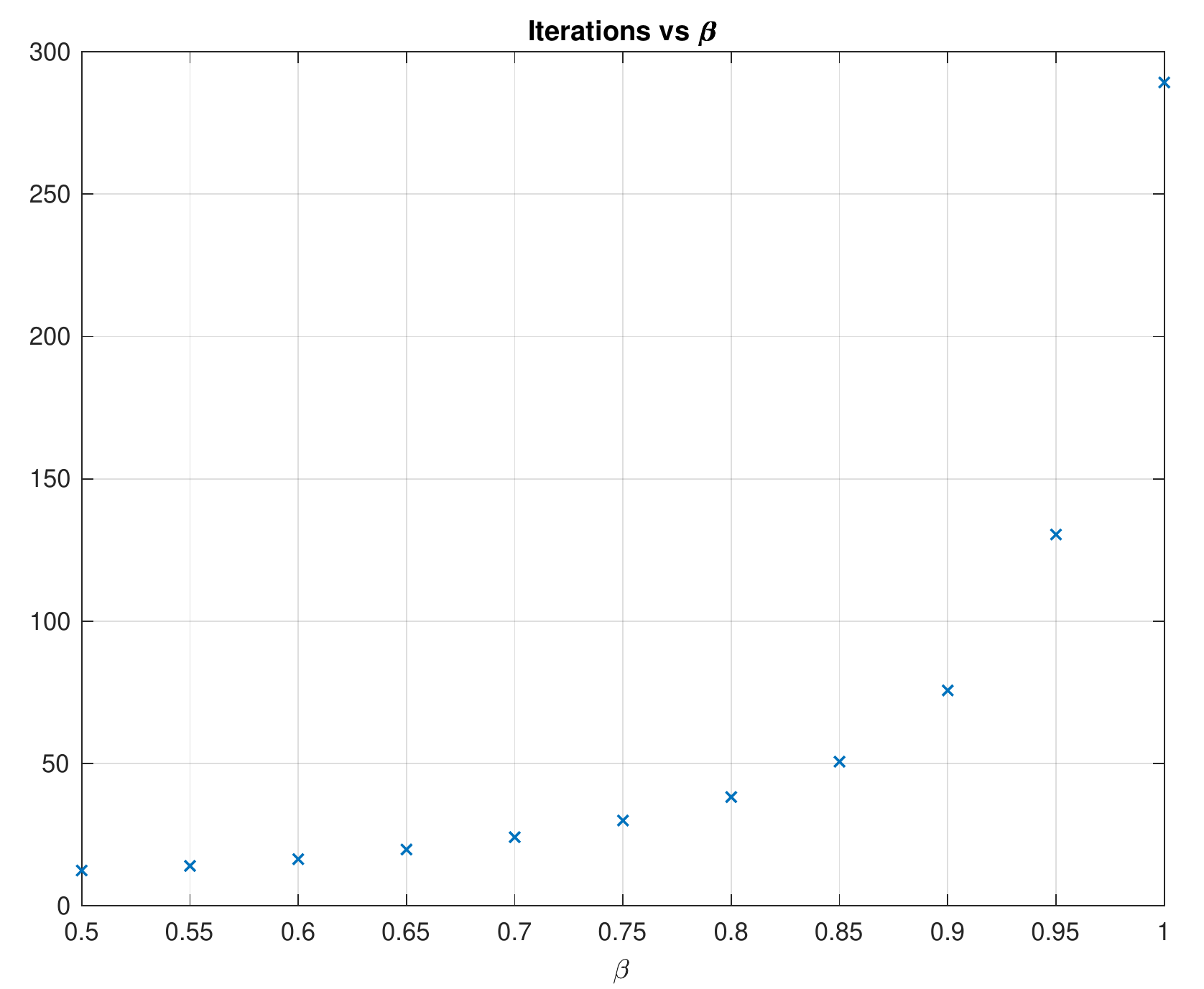}
  \caption{Centralized subgradient method}
  \label{fig2:sub1}
\end{subfigure}%
\begin{subfigure}{2.7in}
  \centering
  \includegraphics[width=2.5in]{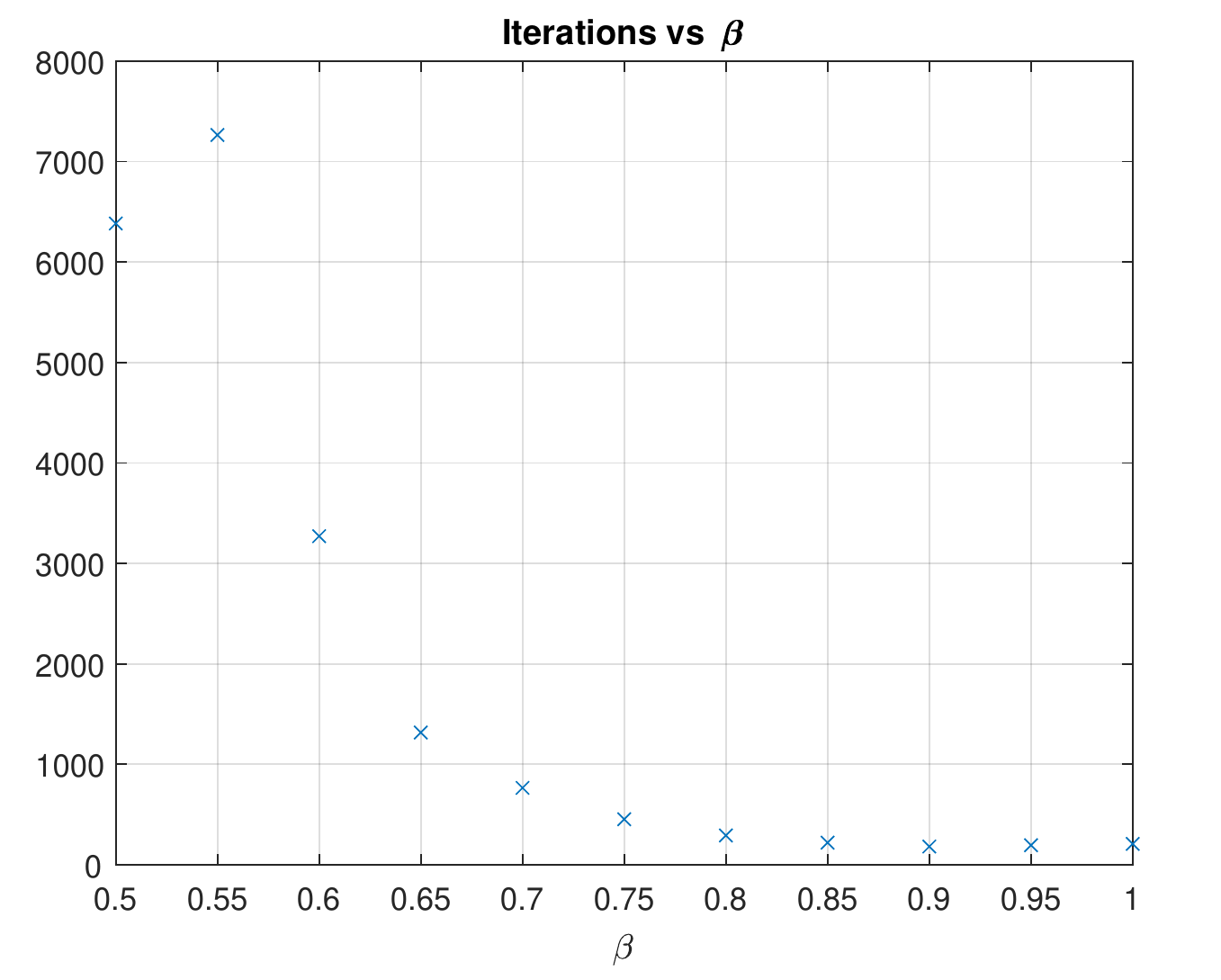}
  \caption{Distributed subgradient method}
  \label{fig2:sub2}
\end{subfigure}
\caption{The relationship between the step-size $1/t^{\beta}$ and the number of iterations needed until the gradient mapping $s(t)$ is small.}
\label{fig2}
\end{figure}

Our starting point is that we want to consider a sufficiently simple class of problems to which it makes sense to apply the subgradient method. Thus, the problem should be non-smooth. On the other hand, it should not be something so simple as a quadratic with an $||\cdot||_1$ regularizer, as that can be solved via proximal methods. A natural choice is to consider problems with an ``elastic net'' regularization \cite{zou2005regularization}, which a sum $||\cdot||_1$ and $||\cdot||_2$ regularizers. For our objective, we consider a regression-like problem where squares are replaces by fourth-powers, so that the ultimate objective is 
\[ F(\theta) = \sum_{i=1}^K ||a_i^T \theta - b_i||_2^4 + \lambda_1 ||\theta||_2 + \lambda_2 ||\theta||_1. \] The logic of using fourth powers, or any exponent higher than $2$, is that the resulting regression problem is very sensitive to large deviations -- at the cost of being less sensitive to small errors. It also makes the problem more difficult for the subgradient method (and if the exponents were quadratic, there would actually be no good reason to solve this problem with a subgradient method in the first place). 

We generate the matrix $A$ which stacks up the vectors $a_i$ to be Gaussian with unit variance; each $a_i$ belongs to $\R^2$. We set $b_i = a_i^T {\bf 1} + w_i$, where $w_i$ is a Gaussian random variable of small variance $1/25$. We choose $\lambda_1=1, \lambda_2=1/20$.

The results are shown in Figure \ref{fig2}. We use $n=10$ agents on the line graph. The two figures plot $\beta$ vs the number of iterations.  Each data point is an average over 10,000 runs; we terminate when  $||s(t)||_1 < 0.03$. A slight difficulty comes from the fact that we have considerable freedom to choose the subgradient of $||\cdot||_1$ at the origin and, somewhat problematically, round-off error will preclude any component of $s(t)$ from being zero exactly. A reasonable fix is compute $||s(t)||_1$ using the best possible choice of subgradient of $||\cdot||_1$, and to treat entries of $s(t)$ in a small enough interval around zero as indistinguishable from zero. 

Note that, for the distributed method, the number of iterations we are plotting in Figure \ref{fig2} is simply the first $t$ in Eq. (\ref{eq:no}) until the termination condition $||s(t)||_1 < 0.03$ is met. For the centralized method, each iteration consists of computing the gradient of the full function $F(t)$. In other words, a single ``iteration'' of the centralized method has the same cost in terms of the number of gradient computations as a single iteration of the distributed ``method'' (both involve computing all $n$ subgradients of the functions $f_1(\cdot), \ldots f_n(\cdot)$). Of course, such a definition does not fully capture the benefits of the distributed method, which computes the gradients in parallel. An alternative definition is possible, where an iteration of the centralized method might be defined as computing the subgradient of a single function $f_i(\cdot)$ with a subgradient step taken once all $n$ of the subgradients have been computed;  this would result in multiplying the y-axis of Figure \ref{fig2:sub1}, which described the performance of the centralized subgradient method, by $n=10$.

Glancing at Figure \ref{fig2:sub1}, we see that the performance of the centralized subgradient method is as expected: higher $\beta$ leads to slower performance. On the other hand, the behavior of the distributed method in Figure \ref{fig2:sub2} is more complicated, with running times that are largest slighty above $1/2$ and decaying thereafter. We do see that it may be possible for $\beta_1$ to be better than $\beta_2$ in the centralized case, only to have the performance flip in the distributed case, which is the ``step-size'' inversion from the title of this section. 

We end this discussion with two caveats. First, this behavior is not true for all practical examples; our point is that step-size inversion {\em can} happen in measures of average-case performance, which is something that is hinted by our results but has not been observed before to our knowledge. In our simulations, step-size inversion  seems to be tied to the oscillations of the subgradient when the optimal solution occurs at a point of non-smoothness (i.e., when some component of $\theta^*$ equals zero in the above example). This is how the parameter values described above were chosen. We wanted to choose values that resulted in occasional solutions with components close to zero, but we didn't want this to happen too often (as then the simulation would take too long, as these are exactly the cases when it takes a very long time converge), which is why we chose $\beta=0.1$. Similarly, the choice of $0.03$ for the $1$-norm of $s(t)$ in the termination condition was picked to achieve a reasonable running time (which is less than overnight on MATLAB running on an iMac; initially, we tried to set that threshold to $0.01$, but the resulting simulations took too long -- recall that each data point in Figure \ref{fig2} is the average over 10,000 runs). 

\section{Conclusion\label{sec:concl}} 

Our goal was to understand when one can obtain network independence and a linear speedup in the distributed subgradient method when compared to its centralized counterpart. We showed that this is possible when the step-size decays like $1/t^{\beta}$ when $\beta > 1/2$, but not when $\beta=1/2$. The bounds we derived on the transient time until the linear speedup kicks in increased as $\beta$ decreased, suggesting that it might be the case that faster-decaying step-sizes are better in the distributed case, even when the opposite is true in the centralized case. We simulated one class of problems with random data where this was indeed the case for a range of step-sizes. 

Our results point to a number of open questions. First, we have no theory to explain the performance we see in Figure \ref{fig2}. The shape of that figure is particularly interesting, as increasing $\beta$ starting from $\beta=1/2$ seems to make things worse before making them better. Second, it would be interesting to conduct a large scale computational study across many problems of interest to see if the step-size inversion phenomenon holds more broadly than was demonstrated here. Since this is a primarily theoretical paper, it is  out of scope for the present work. 

Most importantly, it would be interesting to examine whether a linear speedup can be obtained for the optimal step-size decay $\beta=1/2$ using a different algorithm (in the standard model of distributed optimization where a single message exchange and a single subgradient computation are possible at each node at each step). We are not aware of any lower bounds ruling this out, nor any algorithms  known to achieve this. 
\vskip 0.2in

\bibliographystyle{unsrt}
\bibliography{refs}

\end{document}